\documentclass{amsart}
\usepackage{wrapfig}
\usepackage[dvips]{graphicx}
\usepackage{amsmath,amsthm,amssymb,amscd}
\usepackage{braket}
\theoremstyle{definition}
\newtheorem{thm}{Theorem}[section]
\newtheorem{Def}[thm]{Definition}
\newtheorem{pro}[thm]{Proposition}

\newtheorem{lem}[thm]{Lemma}
\newtheorem{ex}[thm]{Example}
\newtheorem{rem}[thm]{Remark}
\theoremstyle{definition}
\def\c#1{\mathcal {#1}}
\def\m#1{\mathbb {#1}}
\def\r#1{\rm {#1}}
\begin{document}

\title{Fundamental group of finite von Neumann algebras with finite dimensional normal trace space}
\author{Takashi Kawahara}
\address[Takashi Kawahara]{Doctoral student of Mathematics, 
Kyushu University, Ito, 
Fukuoka, 819-0395,  Japan}
\email{t-kawahara@math.kyushu-u.ac.jp}      
\maketitle
\begin{abstract}
We introduce the fundamental group $F(\c{M})$ of 
a finite von Neumann algebra $\c{M}$ with finite dimensional normal trace space.  
The form of $F(\c{M})$ is completely determined.  
Moreover, there exists a finite von Neumann algebra $\c{M}$ with finite dimensional normal trace space such that $F(\c{M})=G$
for any conceivable groups $G$ as fundamental group.  
\end{abstract}

\section{Introduction}
  We recall some facts of fundamental groups of operator algebras.  \\  
The fundamental group $F(M)$ of a ${\rm II}_1$-factor $M$ 
with a normalized trace $\tau$ is defined by Murray and von Neumann in \cite{MN}.  
In the paper, the fact that if $M$ is  hyperfinite, then 
$F(M) = {\mathbb R_+^{\times}}$ is shown, where $\mathbb{R}_+^{\times}$ is the set of positive invertible real numbers.   
It was proved that $F(L(\mathbb{F}_{\infty}))$ of the group factor 
of the free group $\mathbb{F}_{\infty}$ contains the positive rationals by Voiculescu in \cite{Vo} and 
it was shown that 
$F(L(\mathbb{F}_{\infty})) = {\mathbb R}_+^{\times}$ by Radulescu in 
\cite{RF}.  The fact that $F(L(G))$ is a countable group if $G$ is an ICC group with property (T) is shown by Connes \cite{Co}.  That either countable subgroup of $\mathbb R_+^{\times}$ or any uncountable group belonging to a certain "large" class 
can be realized as the fundamental group of some factor of type $\r{II}_1$ is shown by Popa and Vaes in \cite{Po} and in \cite{PoVa}.  

\ Nawata and Watatani \cite{NY}, \cite{NY2} introduced the fundamental group of simple $C^*$-algebras 
with unique trace.  Their study is essentially based on the computation  of 
Picard groups by Kodaka \cite{kod1}, \cite{kod2}, \cite{kod3}.  
Nawata defined the fundamental group of non-unital $C^{*}$-algebras \cite{N1} and calculated the Picard group of some projectionless $C^{*}$-algebras with strict comparison by fundamental group \cite{N2}.
In this paper, we define the fundamental group of $C^*$-algebras 
with finite dimensional trace space.  
This fundamental group is a "numerical invariant".  
Let $\c{M}$ be a finite von Neumann algebra $\c{M}$ with finite dimensional trace space.  
We define the fundamental group $F(\c{M})$ of $\c{M}$  by using self-similarity and the extremal points of the trace state space of $\c{M}$.  
Then $F(\c{M})$ is a 
multiplicative subgroup of $GL_{n}(\m{R})$.  
We shall determine the form of $F(\c{M})$ completely.  
As in the same with the case of the fundamental group of $C^*$-algebras with finite dimensional trace space,  
if the finite von Neumann algebras $\c{M}$ and $\c{N}$ with finite dimensional normal trace space are Morita equivalent, then $F(\c{M})=(DU(\sigma))^{-1}F(\c{N})(DU(\sigma))$ for some diagonal matrix $D$ and for some permutation unitary $U(\sigma)$.  
Popa showed that there exists a ${\rm II}_{1}$-factor $\c{M}_{i}$ such that $F(\c{M}_{i})=G_{i}$ and that $\c{M}_{i}$ is not stably isomorphic each other in \cite{Po}.
By using the fact, we will show that there exists a finite von Neumann algebra $\c{M}$ with finite dimensional normal trace space such that $F(\c{M})=G$
for any conceivable groups $G$ as fundamental group.  
Let $\c{A}$ be a unital $C^{*}$-algebra with 2-dimensional trace space.  
Put $\varphi=\cfrac{1}{2}(\varphi_{1}+\varphi_{2})$, where $\varphi_{1}$ and $\varphi_{2}$ are extreme points of $T(\c{A})$.  
Then $\overline{\pi_{\varphi}(\c{A})}^{w}$ is a finite von Neumann algebra with finite dimensional normal trace space and $F(\c{A})\subset F(\overline{\pi_{\varphi}(\c{A})}^{w})$.  
\\

\section{finite von Neumann algebra and trace space}
We recall some facts of von Neumann algebra and its trace space.  
Let $\c{M}$ be a finite von Neumann algebra.  
We denote by $T(\c{M})$ the set of tracial states of $\c{M}$ and by ${\rm lin}_{\m{C}}T(\c{M})$ the $\m{C}$-linear span of $T(\c{M})$.  
Similarly, we denote by $NT(\c{M})$ the set of normal tracial states of $\c{M}$ and by ${\rm lin}_{\m{C}}NT(\c{M})$ the $\m{C}$-linear span of $NT(\c{M})$.  
Moreover, we denote by $\c{Z}(\c{M})$ the center of $\c{M}$.  
\begin{lem}\label{lem:finite}
Let $\c{M}$ be a finite von Neumann algebra.  
Then the following conditions are equivalent.  \\
(1) $\c{M}$ is a finite direct sum of finite factors.  \\
(2) ${\rm dim}\c{Z}(\c{M})<\infty$ \\
(3) ${\rm dim}({\rm lin}_{\m{C}}T(\c{M}))<\infty $ \\
(4) ${\rm dim}({\rm lin}_{\m{C}}NT(\c{M}))<\infty $
\end{lem}
\begin{proof}
It is clear that (1) and (2) are equivalent, (1) implies (3).  
Since ${\rm lin}_{\m{C}}T(\c{M})$ is a subspace of ${\rm lin}_{\m{C}}NT(\c{M})$, (3) implies (4).  
We shall show that (4) implies (2).  
We assume that (4).  
Let $(X,\mu)$ be a locally finite measure space such that $\c{Z}(\c{M})\cong L^{\infty}(X,\mu)$.  
On the contrary, we assume that ${\rm dim}\c{Z}(\c{M})=\infty$.  
Since $L^{\infty}(X,\mu)$ is isomorphic to the dual space of $L^{1}(X,\mu)$, $L^{1}(X,\mu)$ is not
finite dimensional.  
Let $\theta$ be a unique center-valued trace of $\c{M}$.  
Then $\set{i(f)\circ \theta:f\in L^{1}(X,\mu)}$ is a subspace of ${\rm lin}_{\m{C}}NT(\c{M})$, 
where $i$ is a canonical injection from $L^{1}(X,\mu)$ into the dual space of $L^{\infty}(X,\mu)$.  
Since $L^{1}(X,\mu)$ is not finite dimensional, $\set{i(f)\circ \theta:f\in L^{1}(X,\mu)}$ is not finite dimensional.  
Hence ${\rm lin}_{\m{C}}NT(\c{M})$ is not finite dimensional.  
\end{proof}
 
Especially, if $\c{M}$ is a finite direct sum of finite factors, 
then $T(\c{M})$=$NT(\c{M})$ and ${\rm lin}_{\m{C}}T(\c{M})={\rm lin}_{\m{C}}NT(\c{M})$, 
because every tracial state of finite factors is normal.  

\section{equivalence bimodule of von Neumann algebras}
We recall some definitions of equivalence bimodule of von Neumann algebras (See \cite{R1}).  
Let $\c{M}$ and $\c{N}$ be von Neumann algebras.  
A linear space $\c{X}$ is called a $right\ (left)\ Hermitian\ \c{N}$-$rigged\ \c{M}$-$module$ 
if $\c{X}$ is a right(left) Hilbert $C^{*}$-module over $\c{N}$ with a left (right) action of $\c{M}$ 
by means of a ${}^{*}$-homomorphism of $\c{M}\rightarrow \c{L}_{\c{N}}(\c{X})$ such that $\c{MX}(\c{XM})$ is dense in $\c{X}$ with respect to a norm on $\c{X}$.  
We say a Hermitian right (left) $\c{N}$-rigged $\c{M}$-module is normal if the map $m\mapsto \braket{x,my}_{\c{N}}$ ($m\mapsto {}_{\c{N}}\braket{xm,y}$) is weakly continuous 
for any $x,y$ in $\c{X}$.  
We say $\c{X}$ is a $\c{M}$-$\c{N}$ equivalence bimodule if $\c{X}$ is a normal right $\c{N}$-rigged $\c{M}$-module with a right $\c{N}$-valued inner product $\braket{\ ,\ }_{\c{N}}$ 
and left $\c{M}$-rigged $\c{N}$-module with a left $\c{M}$-valued inner product ${}_{\c{M}}\braket{\ ,\ }$ and satisfies the following conditions; ${}_{\c{M}}\braket{x,y}z=x\braket{y,z}_{\c{N}}$, a linear span of $\{{}_{\c{M}}\braket{x,y}: x,y\in \c{X} \}$ and $\{\braket{x,y}_{\c{N}}: x,y\in \c{X} \}$ are weakly dense in $\c{M}$ and $\c{N}$ respectively.   
We say a $\c{M}$-$\c{N}$ equivalence bimodule $\c{X}$ is $finitely\ generated$ if $\c{X}$ has a finite left and a finite right basis.   
Let $\c{X}$ and $\c{Y}$ be ${\mathcal M}$-${\mathcal N}$ equivalence bimodules.  
We say that $\c{X}$ is isomorphic to $\c{Y}$ as a equivalence bimodule 
if there exists a linear bijective map $\Phi$ from $\c{X}$ onto $\c{Y}$ 
with the properties such that $\Phi(a\xi b)=a\Phi(\xi)b$, 
${}_{\c{M}}\langle \Phi(\xi),\ \Phi(\eta)\rangle 
={}_{\c{M}}\langle \xi,\eta\rangle$ 
and $\langle \Phi(\xi),\ \Phi(\eta)\rangle_{\c{N}}
=\langle \xi, \eta\rangle_{\c{N}}$ 
for any $a$ in $\c{M}$, 
for any $b$ in $\c{N}$ 
and for any $\xi, \eta$ in $\c{X}$.

\section{definition in the case of von Neumann algebras}
\begin{Def}\label{def:s.s}
Let ${\mathcal M}$ be von Neumann algebra.  
A \emph{self-similar triple} for $\c{M}$, abbreviated s.s.t., is a tuple $(k, p, \Phi)$, where $k$ is a natural number, where $p$ is a projection in $M_k(\c{M})$, and where $\Phi$ is an isomorphism from $\c{M}$ onto $pM_k(\c{M})p$.
\end{Def}

We define $T_{(k,p,\Phi)}:{\rm lin}_{\m{C}}T(\c{M})\rightarrow {\rm lin}_{\m{C}}T(\c{M})$ 
by $T_{(k,p,\Phi)}(\varphi)={\rm Tr}_{k}\otimes \varphi \circ \Phi$.  
We denote by $\mathcal{L}({\rm lin}_{\m{C}}T(\c{M}))$ the set of bounded linear maps from ${\rm lin}_{\m{C}}T(\c{M})$ into ${\rm lin}_{\m{C}}T(\c{M})$. 
\begin{Def}\label{def:ontracesp}
Let $\c{A}$ be a unital $C^{*}$-algebra.  
We define a subset $F^{tr}({\mathcal M})$ of $\c{L}({\rm lin}_{\m{C}}T(\c{M}))$ as follows;
\[ F^{tr}({\mathcal M}):=\{ T_{(k,p,\Phi)}\in \c{L}({\rm lin}_{\m{C}}T(\c{M})): (k,p,\Phi ): \r{ s.s.t}\}\]
\end{Def}

We suppose that $\c{M}$ is a finite von Neumann algebra with finite dimensional trace space.  
Then, by using $\c{M}$-$\c{M}$ equivalence bimodules, we will show $F^{tr}(\c{M})$ is a group.  

\begin{lem}\label{lem:full}
Let $\c{M}$ be a a finite von Neumann algebra with finite dimensional trace space 
and let $\c{X}$ be an $\c{M}$-$\c{M}$ equivalence bimodule as $W^{*}$-algebra.  
Then $\c{X}$ has a finite basis.  
\end{lem}
\begin{proof}
Let $I$ be a norm-closed linear span of the set $\set{{}_{\c{M}}\braket{x,y}:x,y\in \c{X}}$.  
Then $I$ is an norm-closed ideal of $\c{M}$.   
By \ref{lem:finite}, $\c{M}$ is a direct sum of finite factors.  
Since the weak${}^{*}$-closed linear span of the set $\set{{}_{\c{M}}\braket{x,y}:x,y\in \c{X}}$ is $\c{M}$, 
$I=\c{M}$.  
Then there exist some elements $x_{i},y_{i}$ in $\c{X}$ such that $\parallel \sum^{n}_{i=1}{}_{\c{M}}\braket{x_{i},y_{i}}-1_{\c{M}}\parallel <1$.  
Therefore there exists an element $m$ in $\c{M}$ such that $m(\sum^{n}_{i=1}{}_{\c{M}}\braket{x_{i},y_{i}})=\sum^{n}_{i=1}{}_{\c{M}}\braket{mx_{i},y_{i}}=1_{\c{M}}$.
For simplicity of notation, we write $x_{i}$ instead of $mx_{i}$.  
Then we can assume that $\sum^{n}_{i=1}{}_{\c{M}}\braket{x_{i},y_{i}}=1_{\c{M}}$ for some $x_{i}$, $y_{i}$.  \\
Put $n=\sum^{n}_{i=1}{}_{\c{M}}\braket{x_{i}+y_{i},x_{i}+y_{i}}$, then $n\geq 1_{\c{M}}$.  \\
Therefore $\sum^{n}_{i=1}{}_{\c{M}}\braket{n^{-\frac{1}{2}}(x_{i}+y_{i}),n^{-\frac{1}{2}}(x_{i}+y_{i})}=1_{\c{M}}$.  
Put $z_{i}=n^{-\frac{1}{2}}(x_{i}+y_{i})$.  
Then $\sum^{n}_{i=1}{}_{\c{M}}\braket{z_{i},z_{i}}=1_{\c{M}}$ and $x=\sum^{n}_{i=1}z_{i}\braket{z_{i},x}_{\c{M}}$ for any $x$ in $\c{X}$.  
Therefore $\set{z_{i}}^{n}_{i=1}$ is a finite basis of $\c{X}$.  
\end{proof}

\begin{pro}\label{pro:respond}
Let $\c{M}$ be a von Neumann algebra
and let $\c{X}$ be an $\c{M}$-$\c{M}$ equivalence bimodule.   
Then there exists a self-similar full projection of $M_{n}(\c{M})$ such that $\c{X}$ is isomorphic to $p\c{M}^{n}$ as an $\c{M}$-$\c{M}$ equivalence bimodule.  
Conversely, $p\c{M}^{n}$ is an $\c{M}$-$\c{M}$ equivalence bimodule if $p$ is a self-similar full projection in $M_{n}(\c{M})$.  
\end{pro}
\begin{proof}
Let $\c{X}$ be an $\c{M}$-$\c{M}$ equivalence bimodule.  
By \ref{lem:full}, there exists a finite basis $\set{z_{i}}^{n}_{i=1}$ of $\c{X}$.  
Put $p=(\braket{z_{i},z_{j}}_\c{M})_{ij} \in M_n(\c{M})$. 
Then $p$ is a full projection and $\mathcal{X}$ is isomorphic to 
$p\c{M}^n$ as $\c{M}$-$\c{M}$ equivalence bimodule 
with an isomorphism of $\c{M}$ to $pM_n(\c{M})p$.  
Conversely, we suppose that $p$ is a full projection in $M_{k}(\c{M})$ with an isomorphism $\alpha:\c{M}\rightarrow pM_{n}(\c{M})p$ .  
Then $p\c{M}^{n}$ is an $\c{M}$-$\c{M}$ equivalence bimodule 
with the operations $m\cdot x=\alpha(m)x$, $x\cdot m=xm$.  
\end{proof}

We denote by $[\c{X}]$ the set of isomorphic classes of $\c{M}$-$\c{M}$ equivalence bimodule $\c{X}$.  
\begin{pro}
The set $\set{[\c{X}] \mid \c{X}:\c{M}\rm{-}\c{M}\ \rm{equivalence\ bimodule}}$ forms a group by internal tensor product.  Identity is $[\c{M}]$.  
\end{pro}

We denote by ${\rm Pic}(\c{M})$ the group.  
We define a map $R_{\c{M}}:{\rm Pic}(\c{M}) \rightarrow \c{L}({\rm lin}_{\m{C}}T(\c{M}))$ by $R_{\c{M}} ([\c{X}])(\varphi)(a)=\sum^{n}_{i=1}\varphi(\braket{\xi_{i},a\xi_{i}}_{\c{M}})$ where $\set{\xi_{i}}^{n}_{i=1}$ is a right basis of $\c{X}$.  

\begin{pro}\label{pro:mul}
The map $R_{\c{M}}$ is well-defined, multiplicative, and $R_{\c{M}}([\c{M}])={\rm Id}_{{\rm lin}_{\m{C}}T(\c{M})}$ where ${\rm Id}_{{\rm lin}_{\m{C}}T(\c{M})}$ is an identity map from ${\rm lin}_{\m{C}}T(\c{M})$ onto ${\rm lin}_{\m{C}}T(\c{M})$.   
\end{pro}

\begin{proof}
Let $\varphi$ be a bounded trace on ${\mathcal M}$, 
$a$ be an element of ${\mathcal M}$, 
${\mathcal E}$ be a finitely generated ${\mathcal M}-{\mathcal M}$ equivalence bimodule and let $\{\xi_{i}\}_{i=1}^{k}$ 
and $\{\eta_{j}\}_{j=1}^{l}$ be finite right bases of ${\mathcal E}$.  
Then 
\begin{eqnarray}
\sum_{i=1}^{k}\varphi(\langle\xi_{i}, a\xi_{i} \rangle_{{\mathcal M}})&=&\sum_{i=1}^{k}\varphi(\langle\xi_{i}, \sum_{j=1}^{l}\eta_{j}\langle\eta_{j}, a\xi_{i} \rangle_{{\mathcal M}} \rangle_{{\mathcal M}}) 
= \sum_{i, j=1}^{k, l} \varphi(\langle \xi_{i}, \eta_{j} \rangle_{{\mathcal M}}\langle \eta_{j}, a\xi_{i} \rangle_{{\mathcal M}}) \nonumber \\
&=& \sum_{i, j=1}^{k, l} \varphi(\langle \eta_{j}, a\xi_{i} \rangle_{M}\langle \xi_{i}, \eta_{j} \rangle_{{\mathcal M}}) 
= \sum_{j=1}^{l}\varphi(\langle\eta_{j}, a\eta_{j} \rangle_{{\mathcal M}})\nonumber 
\end{eqnarray}
Therefore $R_{\mathcal M}([{\mathcal E}])$ is independent on the choice of basis.\\  
\ \ Let ${\mathcal E}_{1}$ and ${\mathcal E}_{2}$ be ${\mathcal M}$-${\mathcal M}$ imprimitivity bimodules 
with bases $\{\xi_{i}\}_{i=1}^{k}$ and $\{\zeta_{j}\}_{j=1}^{l}$.  
We suppose that there exists an isomorphism $\Phi$ of ${\mathcal E}_{1}$ onto ${\mathcal E}_{2}$.  
Then $\{\Phi(\xi_{i})\}_{i=1}^{k}$ is also a basis of ${\mathcal E}_{2}$.  
Then 
\begin{eqnarray}
\sum_{i=1}^{k}\varphi(\langle\xi_{i}, a\xi_{i} \rangle_{{\mathcal M}})=\sum_{i=1}^{k}\varphi(\langle\Phi(\xi_{i}), a\Phi(\xi_{i}) \rangle_{{\mathcal M}})
= \sum_{j=1}^{l}\varphi(\langle\zeta_{i}, a\zeta_{i} \rangle_{{\mathcal M}}) \nonumber 
\end{eqnarray}
Therefore $R_{\mathcal M}$ is well-defined.  \\
\ \ We shall show that $R_{\mathcal M}$ is multiplicative.   
Let ${\mathcal E}_{1}$ and ${\mathcal E}_{2}$ be finitely generated ${\mathcal M}-{\mathcal M}$ equivalence bimodules with bases $\{\xi_{i}\}_{i=1}^{k}$ and $\{\eta_{j}\}_{j=1}^{l}$.  
Then $\{\xi_{i}\otimes\eta_{j}\}_{i, j=1}^{k, l}$ is a basis of ${\mathcal E}_{1}\otimes{\mathcal E}_{2}$ and
\begin{eqnarray}
\left( R_{\mathcal M}([{\mathcal E}_{1}\otimes{\mathcal E}_{2}])\left(\varphi \right) \right) (a)=\sum_{i, j=1}^{k, l}\varphi(\langle \xi_{i}\otimes\eta_{j}, a\xi_{i}\otimes\eta_{j} \rangle)
=\sum_{i, j=1}^{k, l}\varphi(\langle \eta_{j}, \langle\xi_{i}, a\xi_{i} \rangle_{\mathcal M} \eta_{j} \rangle_{{\mathcal M}}) \nonumber
\end{eqnarray}
On the other hand
\begin{eqnarray}
\left(R_{\mathcal M}([{\mathcal E}_{1}])R_{\mathcal M}([{\mathcal E}_{2}])(\varphi)\right)(a) &=& \sum_{i=1}^{k}\left(R_{\mathcal M}([{\mathcal E}_{2}])\left(\varphi \right)\right)(\langle \xi_{i}, a\xi_{i} \rangle_{{\mathcal M}}) \nonumber \\
&=&\sum_{i, j=1}^{k, l}\varphi(\langle \eta_{j}, \langle\xi_{i}, a\xi_{i} \rangle_{{\mathcal M}} \eta_{j} \rangle_{{\mathcal M}}) \nonumber
\end{eqnarray}
Therefore $R_{\mathcal M}$ is multiplicative. 
\end{proof}
\begin{pro}\label{pro:extfundamental group}
Let ${\mathcal M}$ be a von Neumann algebra.  
Then $F^{tr}({\mathcal M})=R_{\mathcal  M}({\rm Pic}({\mathcal M}))$.  
\end{pro}
\begin{proof}
Let ${\mathcal E}$ be a finitely generated ${\mathcal M}$-${\mathcal M}$ equivalence bimodule 
and let $\{\xi_{i}\}^{k}_{i=1}$ be a basis of ${\mathcal E}$.  
Put $p=(\langle \xi_{i},\xi_{j}\rangle)_{ij}$.  
Then ${\mathcal E}$ is isomorphic to $p{\mathcal M}^{k}$ as an ${\mathcal M}$-${\mathcal M}$ equivalence bimodule and there exists an ${}^{*}$-isomorphism $\alpha:{\mathcal M}\rightarrow pM_{k}({\mathcal M})p$.  
Then
\begin{eqnarray}
(R_{{\mathcal M}}([{\mathcal E}])(\varphi ))(a)&=&\sum^{k}_{i=1}\varphi (\langle \xi_{i},a\xi_{i} \rangle )
= \sum^{k}_{i=1}\varphi (\langle pe_{i}, \alpha (a)pe_{i} \rangle )\nonumber\\
&=& \sum^{k}_{i=1}\varphi ( e^{*}_{i}\alpha (a)e_{i})
= ({\rm Tr}_{k}\otimes \varphi)\circ (\alpha)(a)\nonumber
\end{eqnarray}
Therefore $F^{tr}({\mathcal M})\supset R_{{\mathcal M}}({\rm Pic}({\mathcal M}))$.  
Conversely, we suppose that $p$ be a projection with a ${}^{*}$-isomorphism $\alpha:{\mathcal M}\rightarrow pM_{k}({\mathcal M})p$ and that the linear span of $\{a^{*}pb \mid a,b\in {\mathcal M}^{k}\}$ is dense in ${\mathcal M}$.    
Then, $p{\mathcal M}^{k}$ is an ${\mathcal M}$-${\mathcal M}$ equivalence bimodule 
with a basis $\{pe_{i}\}^{k}_{i=1}$.  
Then 
\begin{eqnarray}
({\rm Tr}_{k}\otimes \varphi)\circ (\alpha)(a)&=& \sum^{k}_{i=1}\varphi (\langle e^{*}_{i}, \alpha (a)e_{i} \rangle )
= \sum^{k}_{i=1}\varphi (\langle pe_{i}, \alpha (a)pe_{i} \rangle )\nonumber\\
&=&(R_{{\mathcal M}}([p{\mathcal M}^{k}])(\varphi ))(a)\nonumber
\end{eqnarray}
Therefore $F^{tr}({\mathcal M})\subset$ $R_{{\mathcal M}}({\rm Pic}({\mathcal M}))$.  
Hence $F^{tr}({\mathcal M})$ $=$ $R_{{\mathcal M}}({\rm Pic}({\mathcal M}))$.  
\end{proof}

\begin{pro}
Let $\c{M}$ be a von Neumann algebra.  
Then $F^{tr}(\c{M})$ is a subgroup of $\c{GL}({\rm lin}_{\m{C}}T(\c{M}))$.  
\end{pro}

\begin{proof}
By \ref{pro:mul}, $R_{\c{M}}({\rm Pic}(\c{M}))$ is a subgroup of $\c{GL}({\rm lin}_{\m{C}}T(\c{M}))$.  
Then, by \ref{pro:extfundamental group}, $F^{tr}(\c{M})$ is also.  
\end{proof}

We shall consider the case that $\c{M}$ is a finite von Neumann algebra with finite dimensional normal trace space.  
By \ref{lem:finite}, we can write $\c{M}=\oplus^{n}_{i=1}\c{M}_{i}$ for some finite factor $\c{M}_{i}$.  
Let $\tau_{i}$ be a normalized unique trace of $\c{M}_{i}$.  
We define a normalized trace $\varphi_{i}$ of $\c{M}$ by $\varphi_{i}((x_{k})_{k})=\tau_{i}(x_{i})$.  
Then $\set{\varphi_{i}}^{n}_{i=1}$ is a set of extremal points of $T(\c{M})$ and is a basis of ${\rm lin}_{\m{C}}T(\c{M})$,  
When we use $\set{\varphi_{i}}^{n}_{i=1}$ to denote the basis of $T(\c{M})$, consider that $\varphi_{i}$ corresponds to $M_{i}$.  

\begin{Def}
Let $\c{M}$ be a finite von Neumann algebra with finite dimensional normal trace space.  
Put $\c{M}=\oplus^{n}_{i=1}\c{M}_{i}$, where $\c{M}_{i}$ is a finite factor.   
We define the fundamental group $F(\c{M})$ by the matrix representation of $F^{tr}(\c{M})$ with the basis $\set{\varphi_{i}}^{n}_{i=1}$.  
In other words; We define the representation $S_{{\mathcal M}}:{\mathcal L} ({\rm lin}_{\m{C}}T(\c{M})) \rightarrow M_{n}(\mathbb{C})$
by $S_{{\mathcal M}}(T)(e_{i})=\sum^{n}_{j=1}a_{ij}e_{j}$, 
where $T\varphi_{i}=\sum^{n}_{j=1}a_{ij}\varphi_{j}$, where $e_{i}$ is a canonical basis of $\m{C}^{n}$.  
Then $F(\c{M})=S_{\c{M}}(F^{tr}(\c{M}))$.   
\end{Def}

This fundamental group is a "numerical invariant".  
In the case of ${II}_{1}$-factor, the group is completely numerical invariant.  
In other words, if $F(\c{M})$ and $F(\c{N})$ are different as a set, then $\c{M}$ and $\c{N}$ are not isomorphic (not Morita equivalent).  
This is similar with the case of the $C^{*}$-algebras with finite dimensional trace space.  
Let $\c{M}$ and $\c{N}$ be finite von Neumann algebras with finite dimensional normal trace space. 
The definition of the fundamental group $F(\c{M})$ depends on the permutation of $\partial_{e}T(\c{M})$ if $\sharp(\partial_{e}T(\c{M}))>1$.  
So, we introduce the concept of being isomorphic and that of being weightedly isomorphic on the fundamental groups $F(\c{M})$
considering the fact.  
We define the canonical unitary representation $U$ from a symmetric group $S_{n}$ into $M_{n}(\mathbb{C})$ 
by $U(\sigma)_{ij}=1$ if $j=\sigma(i)$ and $U(\sigma)_{ij}=0$ if $j\neq \sigma(i)$.  

\begin{Def}\label{def:isomorphism}
Let ${\mathcal M}$ and ${\mathcal N}$ be finite von Neumann algebras with finite dimensional normal trace space.  
satisfying $\sharp(\partial_{e}T(\c{M}))=\sharp(\partial_{e}T(\c{N}))=n$.  
We call that $F({\mathcal M})$ is isomorphic to $F({\mathcal N})$ 
if there exists a permutation $\sigma$ in $S_{n}$ 
such that $F({\mathcal N})=(U(\sigma))^{-1}F({\mathcal M})(U(\sigma))$.  
We call that $F({\mathcal M})$ is weightedly isomorphic to $F({\mathcal N})$ 
if there exists an invertible positive diagonal matrix $D$ in $M_{n}(\mathbb{C})$ 
and a permutation $\sigma$ in $S_{n}$ 
such that $F({\mathcal N})=(DU(\sigma))^{-1}F({\mathcal M})(DU(\sigma))$.    
\end{Def}
\begin{pro}\label{pro:Ciso}
If two finite von Neumann algebras, 
which have finite dimensional normal trace space, 
are isomorphic, 
then their fundamental groups are same up to permutation of basis.   
\end{pro}
\begin{proof}
Let ${\mathcal M}$ and ${\mathcal N}$ be finite von Neumann algebras with finite dimensional normal trace space.  
We suppose $\sharp(\partial_{e}T(\c{M}))=n$.  
Say $\{\varphi_{i}\}^{n}_{i=1}=\partial_{e}T(\c{M})$.     
If ${\mathcal M}$ and ${\mathcal N}$ are isomorphic, 
there exists an isomorphism $\alpha:{\mathcal M}\rightarrow{\mathcal N}$.  
Then $\{\varphi_{i}\circ \alpha \}^{n}_{i=1}=\partial_{e}T(\c{N})$.  
\end{proof}

Let $\c{M}$ and $\c{N}$ be finite von Neumann algebras with finite dimensional normal trace space.  
We shall show that if $\c{M}$ and $\c{N}$ are Morita equivalent, 
then their fundamental groups $F(\c{M})$ and $F(\c{N})$ are weightedly isomorphic.  
If ${\mathcal M}$ and ${\mathcal N}$ are Morita equivalent, 
then there exists an ${\mathcal M}$-${\mathcal N}$ equivalence bimodule ${\mathcal F}$.  
We define the linear map 
$R_{{\mathcal M}{\mathcal N}}:{}_{\mathcal M}E_{\mathcal N}\rightarrow {\mathcal L}({\rm lin}_{\m{C}}NT(\c{N}),{\rm lin}_{\m{C}}NT(\c{M}))$ 
by $(R_{{\mathcal M}{\mathcal N}}([{\mathcal F}])(\varphi))(a)
=\sum_{i=1}^{k}\varphi(\langle \xi_{i}, a\xi_{i} \rangle_{{\mathcal M}})$, 
where $\{\xi_{i}\}_{i=1}^{k}$ is a basis of ${\mathcal F}$ as a right Hilbert ${\mathcal N}$-module.    
\begin{lem}\label{lem:hom}
Let ${\mathcal L}$, ${\mathcal M}$ and ${\mathcal N}$ be finite von Neumann algebras with finite dimensional normal trace space. 
We suppose ${\mathcal L}$, ${\mathcal M}$ and ${\mathcal N}$ are Morita equivalent.  
Let ${\mathcal E}$ be an ${\mathcal L}$-${\mathcal M}$ equivalence bimodule 
and let ${\mathcal F}$ be a ${\mathcal M}$-${\mathcal N}$ equivalence bimodule. \\
Then $R_{{\mathcal L}{\mathcal M}}([{\mathcal E}])R_{{\mathcal M}{\mathcal N}}([{\mathcal F}])
=R_{{\mathcal L}{\mathcal N}}([{\mathcal E}\otimes{\mathcal F}])$.  
In particular,  
$R_{{\mathcal L}{\mathcal M}}([{\mathcal F}])R_{{\mathcal M}{\mathcal L}}([{\mathcal F}^*])=id_{{\rm lin}_{\m{C}}T(\c{L})}$, $R_{{\mathcal M}{\mathcal L}}([{\mathcal F}^*])R_{{\mathcal L}{\mathcal M}}([{\mathcal F}])=id_{{\rm lin}_{\m{C}}T({\mathcal M})}$ and $\sharp(\partial_{e}T(\c{L}))=\sharp(\partial_{e}T(\c{M}))$.   
\end{lem}

\begin{proof}
The independence of the choice of the basis of ${\mathcal F}$ and the well-definedness of $R_{\c{L}\c{M}}([\c{F}])$ will be showed similarly 
with the proof of the definition of $R_{\c{L}}([\c{E}])$ in \ref{pro:mul}.   
Moreover, $R_{{\mathcal L}{\mathcal M}}([{\mathcal E}])R_{{\mathcal M}{\mathcal N}}([{\mathcal F}])
=R_{{\mathcal L}{\mathcal N}}([{\mathcal E}\otimes{\mathcal F}])$ will be showed similarly 
with the proof of the multiplicativity of $R_{\c{L}}([\c{E}])$ in \ref{pro:mul}.  
\end{proof}

Especially,  $R_{{\mathcal L}{\mathcal L}}([{\mathcal E}])=R_{\mathcal L}([{\mathcal E}])$ 
where ${\mathcal E}$ is an ${\mathcal L}$-${\mathcal L}$ equivalence bimodule.  

\begin{pro}\label{pro:morita}
Let ${\mathcal M}$ and ${\mathcal N}$ be finite von Neumann algebras with finite dimensional normal trace space. 
If ${\mathcal M}$ and ${\mathcal N}$ are Morita equivalent,  
then $F({\mathcal M})$ is weightedly isomorphic to $F({\mathcal N})$.  
\end{pro}
\begin{proof}
Let ${\mathcal F}$ be an ${\mathcal M}$-${\mathcal N}$ imprimitivity bimodule.  
Then ${\mathcal F}$ induces an isomorphism ${\Psi}$ of $\r{Pic}({\mathcal M})$ to $\r{Pic}({\mathcal N})$ 
such that ${\Psi}([{\mathcal E}])=[{\mathcal F}^*\otimes{\mathcal E}\otimes{\mathcal F}]$ 
for $[{\mathcal E}]\in \r{Pic}({\mathcal M})$.  By \ref{lem:hom}, for $[{\mathcal E}] \in \r{Pic}({\mathcal M})$ 
\begin{eqnarray} 
R_{\mathcal N}([{\mathcal F}^*\otimes{\mathcal E}\otimes{\mathcal F}])&=&R_{{\mathcal N}{\mathcal N}}([{\mathcal F}^*\otimes{\mathcal E}\otimes{\mathcal F}]) \nonumber \\
&=&R_{{\mathcal N}{\mathcal M}}([{\mathcal F}^*])R_{{\mathcal M}{\mathcal M}}([{\mathcal E}])R_{{\mathcal M}{\mathcal N}}([{\mathcal F}]) \nonumber \\
&=&R_{{\mathcal N}{\mathcal M}}([{\mathcal F}^*])R_{\mathcal M}([{\mathcal E}])R_{{\mathcal M}{\mathcal N}}([{\mathcal F}]) \nonumber
\end{eqnarray}
Put $\sharp(\partial_{e}NT(\c{M}))=\sharp(\partial_{e}NT(\c{N}))=n$.  \ 
We \ shall \ consider the representation of the previous formula into $M_{n}(\mathbb{C})$ 
with the basis $\partial_{e}T(\c{M})$ of ${\rm lin}_{\m{C}}T({\mathcal M})$ 
and with the basis $\partial_{e}T(\c{N})$ of ${\rm lin}_{\m{C}}T({\mathcal N})$.  
Then the representation matrices of $R_{\mathcal M}([{\mathcal E}])$ 
and $R_{\mathcal N}([\Psi({\mathcal E})])$ are $S_{\mathcal M}(R_{\mathcal M}([{\mathcal E}]))$ 
and $S_{\mathcal N}(R_{\mathcal N}([\Psi({\mathcal E})]))$ respectively.  
By the definition of $R_{{\mathcal M}{\mathcal N}}({\mathcal F})$ 
and $R_{{\mathcal N}{\mathcal M}}({\mathcal F}^*)$, 
they preserve positiveness of traces, 
so the entries of the matrices which are represented by them are positive.  
Moreover each of them is the inverse element of the other by \ref{lem:hom}.  
Therefore the representation matrix of $R_{{\mathcal A}{\mathcal B}}({\mathcal F})$ satisfies the hypothesis of lemma 4.5 of \cite{TK}.  
Hence there exists a positive invertible diagonal matrix $D$ 
and exists $\sigma$ in $S_{n}$ 
such that $S_{\mathcal N}(R_{\mathcal N}([\Psi({\mathcal E})]))
=(DU(\sigma))^{-1}S_{\mathcal M}(R_{\mathcal M}({\mathcal F}))(DU(\sigma))$.  \\
\end{proof}

\section{The form of fundamental groups}

Let $\c{M}$ be a finite von Neumann algebra with finite dimensional normal trace space.  
Then $\c{M}$ is a direct sum of ${\rm II}_{1}$-factors and matrices by \ref{lem:finite}
Therefore we consider the case that $\c{M}$ is a direct sum of ${\rm II}_{1}$-factors.  
  
Let $\c{M}$ be a finite direct sum of ${\rm II}_{1}$-factors $\c{M}_{i}$.  
We consider the form of the fundamental group $F(\c{M})$.  
As in the same with \cite{TK}, we can show that for any $X$ in $F(\c{M})$ there exists a diagonal $D$ and a permutation unitary $U$ such that $X=DU$.  
However, in this case, we can show the fact more easily and we can show more detailed form of the elements of fundamental groups.  
\begin{lem}\label{lem:multi0}
Let $\c{M}$ be a finite direct sum of ${\rm II}_{1}$-factors $\c{M}_{i}$ 
and let $\varphi_{i}$ be a normalized trace on $\c{M}_{i}$.  
If $X\in F(\c{M})$, then there exists a diagonal $D$ and a permutation unitary $U$ such that $X=DU$.  \\
Put $G_{i}=\set{{\rm Tr}_{n}\otimes\varphi_{i}(p_{i})\mid p_{i}:{\rm  self-similar\ projection\ of}\  M_{n}(\c{M}_{i})}$\\
and put $S_{ij}=\set{{\rm Tr}_{n} \otimes \varphi_{j}(p_{j})\mid p_{j}: {\rm projection\ of} \ M_{n}(\c{M}_{j})\ {\rm s.t.} \ \c{M}_{i}\cong p_{j}M_{n}(\c{M}_{j})p_{j}}$.  
Then $G_{i}=\set{X_{ii}\mid X\in F(\c{M})\ X_{ii}\neq 0}$, $S_{ij}=\set{X_{ij}\mid X\in F(\c{M})\ X_{ij}\neq 0}$.  
\end{lem}
\begin{proof} 
Let $p$ be a self-similar full projection of $M_{n}(\c{M})$.  
Put $p=\oplus^{n}_{i=1}p_{i}$ where $p_{i}$ is a projection of $M_{n}(\c{M}_{i})$.  
Then there exists a permutation $\sigma$ on $\set{1,2,\cdot \cdot \cdot ,n}$ such that $\c{M}_{\sigma(i)}\cong p_{i}M_{n}(\c{M}_{i})p_{i}$.  
We denote by $\Phi_{i}$ the isomorphism from $\c{M}_{i}$ onto $p_{\sigma(i)}M_{n}(\c{M}_{\sigma(i)})p_{\sigma(i)}$.  
Then $({\rm Tr}_{n}\otimes\varphi_{\sigma(i)})\circ \Phi_{i}=\lambda_{i}\varphi_{i}$ for some $\lambda_{i}$.  
Therefore $X_{i\sigma(i)}=\lambda_{i}$ and $X=DU_{\sigma}$, where $D$ is a diagonal matrix which satisfies $D_{ii}=\lambda_{i}$.  
Since ${\rm Tr}_{n}\otimes \varphi_{\sigma(i)}(p_{\sigma(i)})=R_{M}([p\c{M}^{n}])(\varphi_{i})(1)=\lambda_{i}$ by \ref{pro:extfundamental group}, ${\rm Tr}_{n}\otimes \varphi_{\sigma(i)}(p_{\sigma(i)})=X_{i\sigma(i)}$.   
\end{proof}

\begin{lem}\label{lem:multi1}
Let $\c{M}$, $\c{M}_{i}$, $\varphi_{i}$, $G_{i}$, $S_{ij}$ as above.  
If $S_{ij}$ and $S_{jk}$ are nonempty for some $i$, $j$, $k$, then $S^{-1}_{ij}=S_{ji}$ and  $S_{ij}S_{jk}=S_{ik}$.  
\end{lem}

\begin{proof}
By definition of $S_{ij}$.  
\end{proof}

\begin{lem}\label{lem:multi2}
Let $\c{M}$, $\c{M}_{i}$, $\varphi_{i}$, $G_{i}$, $S_{ij}$ as above.  
Then $G_{i}$ is a group.  
Moreover, if $S_{ij}$ is nonempty for some $i$, $j$, then $G_{i}S_{ij}=S_{ij}$, $G_{i}=G_{j}$, and there exists a positive real number $r_{j}$ such that $S_{ij}=r_{j}G_{i}$.  
\end{lem}

\begin{proof}
Since $F(\c{M}_{i})=G_{i}$, then $G_{i}$ is a group.  
We suppose $S_{ij}$ is nonempty for some $i$, $j$.  
Since $F(\c{M})$ is a group, $G_{i}S_{ij}\subset S_{ij}$.  
Moreover $G_{i}S_{ij}\supset S_{ij}$ because $I_{n}\in F(\c{M})$ and $1\in G_{i}$.  
Since the group $G_{i}$ operates the set $S_{ij}$, there exist positive real numbers $r_{\alpha}$ such that $S_{ij}=\sqcup_{\alpha} r_{\alpha}G_{i}$ by orbit decomposition.  
By $G_{i}S_{ij}=S_{ij}$, $S^{-1}_{ij}S_{ij}=G_{i}$. Then there exists a real positive number $r$ such that  $S_{ij}=rG_{i}$.  
Moreover, $G_{i}=S^{-1}_{ij}S_{ij}=S_{ji}S^{-1}_{ji}=G_{j}$ by \ref{lem:multi1}.   
\end{proof}

By \ref{lem:multi0}, \ref{lem:multi1}, and \ref{lem:multi2}, the next proposition follows.  

\begin{pro}\label{pro:2dim}
Let $\c{M}_{1}$ and $\c{M}_{2}$ be ${\rm II}_{1}$-factors.
Put $\c{M}=\c{M}_{1}\oplus \c{M}_{2}$.    
If $S_{12}$ is empty, then for any $X$ in $F(\c{M})$ there exist $g_{i}\in F(\c{M}_{i})$ such that $X=\left[ 
 \begin{array}{cc}
 g_{1} & 0 \\
 0 & g_{2} \\
 \end{array} 
 \right]
$.  If $S_{12}$ is nonempty, then there exists a real positive number $r$ such that for any $X$ in $F(\c{M})$ there exist $g_{i}\in F(\c{M}_{1})$ such that $X=\left[ 
 \begin{array}{cc}
 g_{1} & 0 \\
 0 & g_{2} \\
 \end{array} 
 \right]$ or  $X=\left[ 
 \begin{array}{cc}
 0 & rg_{1} \\
 r^{-1}g_{2} & 0 \\
 \end{array} 
 \right]$.  
\end{pro}

In the case $\c{M}=\oplus^{n}_{i=1} \c{M}_{i}$, similar proposition shall be shown.  

\begin{rem}
This proposition \ref{pro:2dim} can be applied to the fundamental groups of $C^{*}$-algebras.  
Let $\c{A}$ be a unital $C^{*}$-algebras with finite dimensional trace space.  
Put $G_{i}=\set{\lambda \mid \lambda \varphi_{i}= ({\rm Tr}_{n} \otimes \varphi_{i}) \circ \Phi \ \Phi:{\rm self-similar\ map\ of}\  \c{A}}$\\
and put $S_{ij}=\set{\lambda \mid \lambda \varphi_{i}= ({\rm Tr}_{n} \otimes \varphi_{j}) \circ \Phi \ \Phi:{\rm self-similar\ map\ of}\ \c{A}}$.  
Then \ref{lem:multi0}, \ref{lem:multi1}, and \ref{lem:multi2} can be followed similarly.  
Hence following proposition will be showed.    
\end{rem}

\begin{pro}
Let $\c{A}$ be a unital $C^{*}$-algebra with two dimensional trace space.  
If $S_{12}$ is empty, then for any $X$ in $F(\c{A})$ there exist $g_{i}\in G_{i}$ such that $X=\left[ 
 \begin{array}{cc}
 g_{1} & 0 \\
 0 & g_{2} \\
 \end{array} 
 \right]
$.  If $S_{12}$ is nonempty, then there exists a real positive number $r$ such that for any $X$ in $F(\c{A})$ there exist $g_{i}\in G_{1}$ such that $X=\left[ 
 \begin{array}{cc}
 g_{1} & 0 \\
 0 & g_{2} \\
 \end{array} 
 \right]$ or  $X=\left[ 
 \begin{array}{cc}
 0 & rg_{1} \\
 r^{-1}g_{2} & 0 \\
 \end{array} 
 \right]$.  
\end{pro}

\begin{rem}
The proofs of the propositions \ref{lem:multi0}, \ref{lem:multi1}, and \ref{lem:multi2} are based on the internal tensor product and dual module of equivalence bimodules.    
Especially, the proposition \ref{lem:multi2} are based on the next proposition.  
\end{rem} 

\begin{pro}
Let $\c{A}$ and $\c{B}$ be $C^{*}$-algebras.   
For any $\c{A}$-$\c{A}$ equivalence bimodule $\c{E}$ and for any $\c{A}$-$\c{B}$ equivalence bimodule $\c{F}$, there exists an $\c{A}$-$\c{B}$ equivalence bimodule $\c{G}$ such that $\c{F}=\c{E}\otimes \c{G}$.  
\end{pro}

\begin{proof}
Put $\c{G}=\c{E}^{*}\otimes \c{F}$. Then $\c{F}=\c{E}\otimes \c{G}$.
\end{proof}

Moreover, in the case of the direct sum of ${\rm II}_{1}$-factors, the form of the fundamental group are completely determined.  First, we consider the case $n=2$.  

\begin{pro}\label{pro:2dimf}
Let $\c{M}_{1}$ and $\c{M}_{2}$ be ${\rm II}_{1}$-factors.
Put $\c{M}=\c{M}_{1}\oplus \c{M}_{2}$.    
Then one of the two patterns occurs;
\begin{enumerate}
 \item $F(\c{M})=\set{\left[ 
 \begin{array}{cc}
 g_{1} & 0 \\
 0 & g_{2} \\
 \end{array} 
 \right]\mid g_{1}\in F(M_{1})\ g_{2}\in F(M_{2})}$.  
 \item There exists a positive real number $r$ such that \\
$F(\c{M})=\set{\left[ 
 \begin{array}{cc}
 g_{1} & 0 \\
 0 & g_{2} \\
 \end{array} 
 \right],\ \left[ 
 \begin{array}{cc}
 0 & rg_{1} \\
 r^{-1}g_{2} & 0 \\
 \end{array} 
 \right]\mid g_{1}\in F(M_{1})\ g_{2}\in F(\c{M}_{2})}$.  
 In this case,  $F(M_{1})=F(\c{M}_{2})$.  
\end{enumerate} 
\end{pro}

\begin{proof}
By \ref{pro:2dim}, $F(\c{M})\subset \set{\left[ 
 \begin{array}{cc}
 g_{1} & 0 \\
 0 & g_{2} \\
 \end{array} 
 \right]\mid g_{1}\in F(M_{1})\ g_{2}\in F(M_{2})}$ or\\
 $F(\c{M})\subset \set{\left[ 
 \begin{array}{cc}
 g_{1} & 0 \\
 0 & g_{2} \\
 \end{array} 
 \right],\ \left[ 
 \begin{array}{cc}
 0 & rg_{1} \\
 r^{-1}g_{2} & 0 \\
 \end{array} 
 \right]\mid g_{1}\in F(M_{1})\ g_{2}\in F(\c{M}_{2})}$ for some $r$.  
To prove converse inclusion, it is sufficient to show that for any $g_{1}\in F(\c{M}_{1})$ $\left[ 
 \begin{array}{cc}
 g_{1} & 0 \\
 0 & 1 \\
 \end{array} 
 \right]\in F(\c{M})$.  
Let $g_{1}$ be an element of $ F(\c{M}_{1})$.  
Then there exists a self-similar projection $p$ of $M_{n}(\c{M}_{1})$ such that ${\rm Tr}_{n}\otimes \varphi_{1}(p)=g_{1}$.  
Therefor $p\oplus e_{11}$ is a self-similar projection of $M_{n}(\c{M})$ where $e_{11}$ is a projection of $M_{n}(\c{M}_{2})$.  
Hence $\left[ 
 \begin{array}{cc}
 g_{1} & 0 \\
 0 & 1 \\
 \end{array} 
 \right]\in F(\c{M})$.  
\end{proof}

Similarly, we can consider the case $n=3$.   

\begin{pro}
Let $\c{M}_{1}$, $\c{M}_{2}$ and $\c{M}_{3}$ be ${\rm II}_{1}$-factors.
Put $\c{M}=\c{M}_{1}\oplus \c{M}_{2}\oplus \c{M}_{3}$.  
Then, by permutating $\c{M}_{i}$, one of the three patterns occurs;
\begin{enumerate}
 \item $F(\c{M})=\set{\left[ 
 \begin{array}{ccc}
 g_{1} & 0 & 0\\
 0 & g_{2} & 0\\
 0 & 0 & g_{3}\\
 \end{array} 
 \right]\mid g_{1}\in F(M_{1}),\ g_{2}\in F(M_{2}),\ g_{3}\in F(M_{3})}$.  
 \item There exists a positive real number $r$ such that 
$F(\c{M})$ is consisted of the all matrices which have following two forms;
\begin{itemize}
 \item $\left[\begin{array}{ccc}
 g_{1} & 0 & 0\\
 0 & g_{2} & 0\\
 0 & 0 & g_{3}\\
 \end{array} 
 \right]$\ $g_{1}\in F(\c{M}_{1}),\ g_{2}\in F(\c{M}_{2}),\ g_{3}\in F(\c{M}_{3})$
 \item $\left[ 
 \begin{array}{ccc}
 g_{1} & 0 & 0  \\
 0 & 0 & rg_{2} \\
 0 & r^{-1}g_{3} & 0 \\
 \end{array} 
 \right]$\ $g_{1}\in F(\c{M}_{1}),\ g_{2}\in F(\c{M}_{2}),\ g_{3}\in F(\c{M}_{3})$
 \end{itemize} 
 In this case,  $F(\c{M}_{2})=F(\c{M}_{3})$.
 \item There exists positive real numbers $r_{1},r_{2}$ such that 
 $F(\c{M})$ is consisted of the all matrices which have following six forms;
 \begin{itemize}
 \item $\left[ 
 \begin{array}{ccc}
 g_{1} & 0 & 0\\
 0 & g_{2} & 0\\
 0 & 0 & g_{3}\\
 \end{array} 
 \right],\ \left[ 
 \begin{array}{ccc}
 0 & r_{1}g_{1} & 0  \\
 r^{-1}_{1}g_{2} & 0 & 0 \\
 0 & 0 & g_{3} \\
 \end{array} 
 \right],\ \left[ 
 \begin{array}{ccc}
 g_{1} & 0 & 0  \\
 0 & 0 & r_{2}g_{2} \\
 0 & r^{-1}_{2}g_{3} & 0 \\
 \end{array} 
 \right]$
 \item $\left[ 
 \begin{array}{ccc}
 0 & 0 & r_{1}r_{2}g_{1}  \\
 0 & g_{2} & 0 \\
 r^{-1}_{1}r^{-1}_{2}g_{3}& 0 & 0  \\
 \end{array} 
 \right],\ \left[ 
 \begin{array}{ccc}
 0 & 0 & r_{1}r_{2}g_{1}  \\
 r^{-1}_{1}g_{2} & 0 & 0  \\
 0 & r^{-1}_{2}g_{3} & 0  \\
 \end{array} 
 \right]$
 \item $\left[ 
 \begin{array}{ccc}
 0 & r_{1}g_{1} & 0 \\
 0 & 0 & r_{2}g_{2} \\
 r^{-1}_{1}r^{-1}_{2}g_{3}& 0 & 0  \\
 \end{array} 
 \right]$\\
 $g_{1}\in F(\c{M}_{1}),\ g_{2}\in F(\c{M}_{2}),\ g_{3}\in F(\c{M}_{3})$
 \end{itemize} 
 In this case,  $F(\c{M}_{1})=F(\c{M}_{2})=F(\c{M}_{3})$.  
 \end{enumerate} 
\end{pro}
\begin{proof}
We consider $S_{12}$ and $S_{13}$.  
If both of $S_{12}$ and $S_{13}$ are empty, then $F(\c{M})$ is a subgroup of the group of the pattern 1.  
If one of the sets $S_{12}$ and $S_{13}$ is empty and the other is nonempty, by permutating $\c{M}_{i}$, $F(\c{M})$ is a subgroup of the group of the pattern 2.  
If both of $S_{12}$ and $S_{13}$ are nonempty, $F(\c{M})$ is a subgroup of the group of the pattern 3.  
As in the same with \ref{pro:2dimf}, $\left[ 
 \begin{array}{ccc}
 g_{1} & 0 & 0 \\
 0 & 1 & 0 \\
 0 & 0 & 1  \\
 \end{array} 
 \right]$, $\left[ 
 \begin{array}{ccc}
 1 & 0 & 0 \\
 0 & g_{2} & 0 \\
 0 & 0 & 1  \\
 \end{array} 
 \right]$, and $\left[ 
 \begin{array}{ccc}
 1 & 0 & 0 \\
 0 & 1 & 0 \\
 0 & 0 & g_{3}  \\
 \end{array} 
 \right]$ are in $F(\c{M})$ for any $g_{i}\in F(\c{M}_{i})$.  
 Then $F(\c{M})$ agrees with one of the patterns.  
\end{proof}

We consider the case $n\geq 3$.  
Let $\c{M}$ be a direct sum of $n\ {\rm II}_{1}$-factors $\c{M}_{i}$.  
Put $\c{M}=\oplus^{n}_{i=1}\c{M}_{i}$.  
We define an equivalence relation $\sim_{\c{M}}$ on $\set{1,\cdots,n}$ by $\c{M}_{i}\cong p_{j}M_{n}(M_{j})p_{j}$ for some projection $p_{j}$ of $p_{j}M_{n}(M_{j})p_{j}$.  
Let $\set{i_{k}}^{l}_{k=1}$ be a complete system of representatives of this relation.  
We consider the tuple $(\sharp ([i_{k}]_{\c{M}}))^{l}_{k=1}$, where $[j]_{\c{M}}$ is an equivalence class of $j$ by $\sim_{\c{M}}$.  
We call this tuple $size\ tuple$ of $\c{M}$.  
This tuple is invariant for the complete system of representatives up to permutation.  
\begin{pro}
Let $\c{M}$ and $\c{N}$ be direct sums of $n\ {\rm II}_{1}$-factors $\c{M}_{i}$ and $\c{N}_{j}$ respectively.  
Put $\c{M}=\oplus^{n}_{i=1}\c{M}_{i}$ and $\c{N}=\oplus^{n}_{j=1}\c{N}_{j}$.  
We denote by $(\sharp ([i_{k}]_{\c{M}}))^{l_{\c{M}}}_{k=1}$ and $(\sharp ([j_{k}]_{\c{N}}))^{l_{\c{N}}}_{k=1}$ the size tuples of $\c{M}$ and $\c{N}$ respectively.  
If $\c{M}$ and $\c{N}$ are isomorphic (Morita equivalent), then $l_{\c{M}}=l_{\c{N}}$ and then there exists $\sigma\in S_{l_{\c{M}}}$ such that $(\sharp ([\sigma(i_{k})]_{\c{M}}))^{l_{\c{M}}}_{k=1}=(\sharp ([j_{k}]_{\c{N}}))^{l_{\c{N}}}_{k=1}$
\end{pro}
Put $\c{M}=\oplus^{n}_{i=1}\c{M}_{i}$, where $\c{M}_{i}$ is a ${\rm II}_{1}$-factor.  
By this classification, we need only to consider the case: $\c{M}=\oplus^{n}_{i=1}\oplus^{m}_{j=1}\c{M}_{i,j}$, 
where $\c{M}_{i,j_{1}}$ and $\c{M}_{i,j_{2}}$ is stably isomorphic for any $j_{1},j_{2}$, 
and where $\c{M}_{i_{1},j_{1}}$ and $\c{M}_{i_{2},j_{2}}$ is not stably isomorphic if $i_{1}\neq i_{2}$.  
From now, we consider the form of fundamental group of $\c{M}=\oplus^{n}_{i=1}\c{M}_{i}$, where $\c{M}_{i_{1}}$ and $\c{M}_{i_{2}}$ is stably isomorphic for any $i_{1},i_{2}$.  
Put $DU_{n}(\m{R}_{+})=\set{DU_{\sigma}: D\ {\rm is\ a\ positive\ diagonal\ matrix\ in}\ GL_{n}(\m{R})\ \sigma\in S_{n}}$.  
Let $G$ be a subgroup of $\m{R}^{\times}_{+}$, $n$ be in $\m{N}$ and $r_{1},\cdots,r_{n}$ be positive numbers.  
We denote by $\c{G}(G,\set{r_{i}}^{n}_{i=1})$ the group in $DU_{n+1}(\m{R}_{+})$ which is generated by diagonal matrices whose diagonal elements are in $G$ and by unitaries $U^{r_{j}}_{(1,j+1)}$ which satisfies  $(U^{r_{j}}_{(1,j+1)})_{1,j+1}=(U^{r_{j}}_{(1,j+1)})^{-1}_{j+1,1}=r_{j}$, $(U^{r_{j}}_{(1,j+1)})_{l,l}=1$ if $l\neq 1,j+1$, and the other elements of $U^{r_{j}}_{(1,j+1)}$ are $0$.  
In fact, $\c{G}(G,\set{r_{i}}^{n}_{i=1})$ is isomorphic to $G^{n} \rtimes S_{n}$ as group.  

\begin{pro}
Let $\c{M}$ be a direct sum of $n\ {\rm II}_{1}$-factors $\c{M}_{i}$, where $\c{M}_{i_{1}}$ and $\c{M}_{i_{2}}$ are stably isomorphic for any $i_{1},i_{2}$.  
Put $G=F(\c{M}_{1})$.  
Then there exists $r_{j}>0$ such that $F(\c{M})=\c{G}(G,\set{r_{i}}^{n}_{i=1})$.  
\end{pro}

\begin{proof}
Since all the components of the sum are stably isomorphic, $G=G_{i}$ and $S_{ij}\neq 0$ for any $i,j$ by the definition of $G_{i},S_{ij}$ in \ref{lem:multi0}.  
Then, for $j\geq 1$, there exists $\lambda_{j}>0$ such that $S_{1j}=\lambda_{j}G$.  
Since $S_{n}$ is generated by $\set{(1\ j):j=2,\cdots,n}$, $F(\c{M})$ is generated by $U^{\lambda_{j}}_{(1,j+1)}$ and by diagonal matrices whose diagonal entries are in $G$.  
\end{proof}

\begin{pro}
Let $\c{M}$ be a direct sum of ${\rm II}_{1}$-factors.  
By permutating some components of the direct sum, there exist some subgroups $G_{j}$ and $H_{k}$ in $\m{R}^{\times}_{+}$  and some positive real number $r_{i,j}$ such that\\
$F(\c{M})=
\left[ 
 \begin{array}{ccccccc}
 G_{1} & 0      & \cdots & \cdots                                   & \cdots & \cdots & 0      \\
 \vdots& \ddots & \ddots &                                          &        &         & \vdots \\ 
 0     & \cdots & G_{l}  & 0                                        & \cdots & \cdots & 0      \\
 O     & \cdots & O      & \c{G}(H_{1},\set{r_{i,1}}^{n_{1}}_{i=1}) & O      & \cdots & O      \\
 \vdots& \ddots &        & \ddots                                   & \ddots & \ddots & \vdots \\
 \vdots&        & \ddots &                                          & \ddots & \ddots & O      \\
 O     & \cdots & \cdots & O                                        & \cdots & O      &\c{G}(H_{m},\set{r_{i,m}}^{n_{m}}_{i=1}) \\
\end{array} 
\right]$.  \\
However, the part of $G_{j}$ or that of $\c{G}(H_{1},\set{r_{i,1}}^{n_{1}}_{i=1})$ might be empty.  
\end{pro}

\begin{thm}
Let $k$ be a natural number and let $G_{1},\cdots,G_{l}$ and $H_{1},\cdots H_{m}$ be countable subgroups of $\m{R}^{\times}_{+}$.  
For any $n_{1},\cdots,n_{m}$ and for any $r_{i,j}>0$ there exists a direct sum of ${\rm II}_{1}$-factors $\c{M}$ such that \\
$F(\c{M})=
\left[ 
 \begin{array}{ccccccc}
 G_{1} & 0      & \cdots & \cdots                                   & \cdots & \cdots & 0      \\
 \vdots& \ddots & \ddots &                                          &        &         & \vdots \\ 
 0     & \cdots & G_{l}  & 0                                        & \cdots & \cdots & 0      \\
 O     & \cdots & O      & \c{G}(H_{1},\set{r_{i,1}}^{n_{1}}_{i=1}) & O      & \cdots & O      \\
 \vdots& \ddots &        & \ddots                                   & \ddots & \ddots & \vdots \\
 \vdots&        & \ddots &                                          & \ddots & \ddots & O      \\
 O     & \cdots & \cdots & O                                        & \cdots & O      &\c{G}(H_{m},\set{r_{i,m}}^{n_{m}}_{i=1}) \\
\end{array} 
\right]$.  
\end{thm}
\begin{proof}
Let $G_{1},\cdots,G_{k}$ be countable subgroups of $\m{R}^{\times}_{+}$.  
By corollary 5.5 of Popa \cite{Po}, there exists a ${\rm II}_{1}$-factor $\c{M}_{i}$ such that $F(\c{M}_{i})=G_{i}$ and that $\c{M}_{i}$ is not stably isomorphic each other.  
Let $r_{i,j}>0$ arbitrary.  
Then there exists a projection $p_{i,j}$ of $M_{m_{i,j}}(M)$ such that ${\rm Tr}_{m_{i.j}}\otimes\varphi_{i}(p_{i,j})=r_{i,j}$, where $\varphi_{i}$ is a normalized trace on $\c{M}_{i}$.  
Put $\c{M}=\oplus^{k}_{i=1}\oplus^{n_{i}}_{j=1} p_{i,j}M_{m_{i,j}}(M)p_{i,j}$.  
Since $S_{n}$ is generated by $\set{(1\ j):j=2,\cdots,n}$,\\ $F(\c{M})=
\left[ 
 \begin{array}{ccccccc}
 G_{1} & 0      & \cdots & \cdots                                   & \cdots & \cdots & 0      \\
 \vdots& \ddots & \ddots &                                          &        &         & \vdots \\ 
 0     & \cdots & G_{l}  & 0                                        & \cdots & \cdots & 0      \\
 O     & \cdots & O      & \c{G}(H_{1},\set{r_{i,1}}^{n_{1}}_{i=1}) & O      & \cdots & O      \\
 \vdots& \ddots &        & \ddots                                   & \ddots & \ddots & \vdots \\
 \vdots&        & \ddots &                                          & \ddots & \ddots & O      \\
 O     & \cdots & \cdots & O                                        & \cdots & O      &\c{G}(H_{m},\set{r_{i,m}}^{n_{m}}_{i=1}) \\
\end{array} 
\right]$.  
\end{proof}

\begin{ex}
Let $\c{R}$ be a hyperfinite ${\rm II}_{1}$-factor.  
Put $\c{M}=\c{R}\oplus\c{R}$.  
Then $F(\c{M})=\set{\left[ 
 \begin{array}{cc}
 r_{1}  & 0      \\
 0      & r_{2}  \\
\end{array} 
 \right],\left[ 
 \begin{array}{cc}
 0  & r_{3}  \\
 r_{4}  & 0  \\
\end{array} 
 \right];r_{i}\in \m{R}^{\times}_{+}}$.  
\end{ex}

\begin{ex}
Let $\c{R}$ be a hyperfinite ${\rm II}_{1}$-factor 
and let $L(\m{F}_{\infty})$ be a group factor of the free group $\m{F}_{\infty}$.  
Put $\c{M}=\c{R}\oplus L(\m{F}_{\infty})$.  
Then $F(\c{M})=\set{\left[ 
 \begin{array}{cc}
 r_{1}  & 0      \\
 0      & r_{2}  \\
\end{array} 
 \right];r_{i}\in \m{R}^{\times}_{+}}$.  
\end{ex}

\begin{ex}
Put $\c{M}=\oplus^{n}_{i=1} L(\m{Z}^{2}\rtimes SL_{2}(\m{Z}))$.  
Then $F(\c{M})=\set{U_{\sigma}:\sigma \in S_{n}}$.  
In particular, if $n=2$, then $F(\c{M})=\set{\left[ 
 \begin{array}{cc}
 1 & 0  \\
 0 & 1  \\
\end{array} 
 \right],\left[ 
 \begin{array}{cc}
 0  &  1  \\
 1  &  0  \\
\end{array} 
 \right]}$.  
\end{ex}

\section{relation with the fundamental group of $C^{*}$-algebra}
 Let $\c{A}$ be a unital $C^{*}$-algebra with finite dimensional trace space.  
We denote by $T(\c{A})$ the tracial state space of $\c{A}$.  
We say that $T(\c{A})$ is faithful 
if for any positive element $a$ in $\c{A}$, for any $\varphi$ in $T(\c{A})$, $\varphi(a)=0$, then $a=0$.  
Let $\pi_{\varphi}$ be a GNS representation of $\c{A}$ associated with $\varphi$ in $T(\c{A})$ and $\c{H}_{\varphi}$ be a corresponding Hilbert space.\\  
\ For simplicity, we consider the case that dim${\rm lin}_{\m{C}}T(\c{A})=2$.  
Following arguments are similar in the case that dim${\rm lin}_{\m{C}}T(\c{A})>2$.\\  
\ Say $\partial_{e}(T(\c{A}))=\set{\varphi_{1},\ \varphi_{2}}$.  
We define a representation $\pi:\c{A}\rightarrow \c{L}(\c{H}_{\varphi_{1}}\oplus \c{H}_{\varphi_{2}})$ by $\pi(a)=\left[ 
 \begin{array}{cc}
 \pi_{\varphi_{1}}(a) & 0\\
  0 & \pi_{\varphi_{2}}(a) \\
 \end{array} 
 \right]$. 
If $T(\c{A})$ is faithful, then $\pi$ is isometric.  
We denote by $\eta_{\varphi}$ the natural linear map from $\c{A}$ to $\c{H}_{\varphi}$.  
Then $\eta_{\varphi_{1}}(1_{\c{A}})$ and $\eta_{\varphi_{2}}(1_{\c{A}})$ are cyclic vectors of the representation $\pi_{\varphi_{1}}$ and $\pi_{\varphi_{2}}$ respectively.  
Put $\xi=\frac{1}{\sqrt{2}}(\eta_{\varphi_{1}}(1_{\c{A}}),\eta_{\varphi_{2}}(1_{\c{A}}))$ 
and let $\c{H}$ be the closure of $\set{\pi(a)\xi : a\in \c{A}}$.  
\begin{pro}
Let $\c{A}$ be a unital $C^{*}$-algebra with 2-dimensional trace space $\rm{T}(\c{A})$.  
If $T(\c{A})$ is faithful, then $\pi$ is isometric.  
\end{pro}
\begin{proof}
We suppose $\pi(a)=0$.  
Since $\pi(a)\xi=0$, $\eta_{1}(a)=0$ and $\eta_{2}(a)=0$.  
Therefore $\varphi_{1}(a^{*}a)=\varphi_{2}(a^{*}a)=0$.  
Since $T(\c{A})$ is faithful $a^{*}a=0$.  
\end{proof}
Let $\hat \pi:\c{A}\rightarrow \c{L}(\c{H})$ be the restriction of $\pi$.  
\begin{pro}\label{pro:uni}
Put $\varphi=\cfrac{\varphi_{1}+\varphi_{2}}{2}$.  
Then $\pi_{\varphi}$ and ${\hat \pi}$ is unitary equivalent.  
\end{pro}
\begin{proof}
Since $\varphi(a)=\braket{\hat\pi (a)\xi,\xi}_{\c{H}}$, as follows.  
\end{proof}

Let $\c{A}$ be a unital $C^{*}$-algebra and $\varphi$ be a non-zero element of $T(\c{A})$.  
We denote by $\overline{\pi_{\varphi}(\c{A})}^{w} $ the weak-closure of $\pi_{\varphi}(\c{A})$ in $\c{L}(\c{H}_{\varphi})$.  

\begin{pro}\label{pro:disjoint}
Let $\c{A}$ be a unital $C^{*}$-algebra with 2-dimensional trace space.  
We suppose $T(\c{A})$ is faithful.  
Say $\partial_{e}T(\c{A}))=\set{\varphi_{1}, \varphi_{2}}$.  
Put $\varphi=\cfrac{1}{2}(\varphi_{1}+\varphi_{2})$.  
Then $\overline{\pi_{\varphi}(\c{A})}^{w}$ is unitary equivalent to $\overline{\pi_{\varphi_{1}}(\c{A})}^{w}\oplus \overline{\pi_{\varphi_{2}}(\c{A})}^{w}$.  
\end{pro}

\begin{proof}
By \cite{Bla}, $\pi_{\varphi_{1}}$ and $\pi_{\varphi_{2}}$ are quasi-equivalent or disjoint.  
If $\pi_{\varphi_{1}}$ and $\pi_{\varphi_{2}}$ are quasi-equivalent, then there exists an isomorphism $\alpha:\overline{\pi_{\varphi_{1}}(\c{A})}^{w}\rightarrow \overline{\pi_{\varphi_{2}}(\c{A})}^{w}$ such that ${\hat \pi}(\c{A})=\set{\left[ 
 \begin{array}{cc}
 a & 0\\
  0 & \alpha(a) \\
 \end{array} 
 \right]\mid a\in \overline{\pi_{\varphi_{1}}(\c{A})}^{w}}$.  
This contradicts to the fact that $\varphi_{1}$ and $\varphi_{2}$ are extremal points of $T(\c{A})$.  
Then $\pi_{\varphi_{1}}$ and $\pi_{\varphi_{2}}$ are disjoint.  
Therefore $ \overline{{\hat \pi}(\c{A})}^{w}=\overline{\pi_{\varphi_{1}}(\c{A})}^{w}\oplus \overline{\pi_{\varphi_{2}}(\c{A})}^{w}$. 
By \ref{pro:uni}, $\overline{\pi_{\varphi}(\c{A})}^{w}$ is unitary isomorphic to $\overline{\pi_{\varphi_{1}}(\c{A})}^{w}\oplus \overline{\pi_{\varphi_{2}}(\c{A})}^{w}$.  
\end{proof}

\begin{pro}
Let $\c{A}$ be a unital $C^{*}$-algebra with 2-dimensional trace space.  
We suppose $T(\c{A})$ is faithful.  
Say $\partial_{e}T(\c{A})=\set{\varphi_{1}, \varphi_{2}}$.  
Put $\varphi=\cfrac{1}{2}(\varphi_{1}+\varphi_{2})$.  
If $p$ is a self-similar full projection in $M_{n}(\c{A})$, then $\pi_{\varphi}(p)$ in $M_{n}(\overline{\pi_{\varphi}(\c{A})}^{w})$  is a self-similar full projection.  
\end{pro}

\begin{proof}
Let $p$ be a self-similar full projection of $M_{n}(\c{A})$ and let $\Phi$ be an isomorphism from $\c{A}$ onto $pM_{n}(\c{A})p$.  
Then $({\rm Tr}_{n}\otimes \varphi_{1})\circ \Phi=\varphi_{1}$ and $({\rm Tr}_{n}\otimes \varphi_{2})\circ \Phi=\varphi_{2}$ or $({\rm Tr}_{n}\otimes \varphi_{1})\circ \Phi=\varphi_{2}$ and $({\rm Tr}_{n}\otimes \varphi_{2})\circ \Phi=\varphi_{1}$.  
By \ref{pro:disjoint}, we consider the weak topology convergence of $M_{n}(\overline{\pi_{\varphi}(\c{A}})$ as that of $\overline{\pi_{\varphi_{1}}(\c{A})}^{w}\oplus \overline{\pi_{\varphi_{2}}(\c{A})}^{w}$.  
Then $\pi_{\varphi}(p)$ is a self-similar projection.    
\end{proof}

Hence
\begin{pro}
Let $\c{A}$ be a unital $C^{*}$-algebra with 2-dimensional trace space.  
We suppose $T(\c{A})$ is faithful.  
Then $F(\c{A})\subset F(\overline{\pi_{\varphi}(\c{A})}^{w})$.  
\end{pro}

\end{document}